\documentclass{amsart}

\usepackage{amssymb,amsmath,amsthm,latexsym,graphicx}

  \theoremstyle{definition}
  \newtheorem{definition}{Definition}
  
  \newtheorem{remark}[definition]{Remark}

  \theoremstyle{plain}
  \newtheorem{lemma}[definition]{Lemma}
  
  \newtheorem{theorem}[definition]{Theorem}

\begin{document}

\title{On the definition of quasi-Jordan algebra}

\author{Murray R. Bremner}

\address{Department of Mathematics and Statistics, University of Saskatchewan,
Saskatoon, Saskatchewan, Canada}

\email{bremner@math.usask.ca}

\begin{abstract}
Vel\'asquez and Felipe recently introduced quasi-Jordan algebras based on the
product $a \triangleleft b = \tfrac12 ( a \dashv b \,+\, b \vdash a )$ in an
associative dialgebra with operations $\dashv$ and $\vdash$. We determine the
polynomial identities of degree $\le 4$ satisfied by this product.  In addition
to right commutativity and the right quasi-Jordan identity, we obtain a new
associator-derivation identity.
\end{abstract}

\maketitle

Loday \cite{Loday1, Loday2} defined an \textbf{(associative) dialgebra} to be a
vector space with two bilinear operations $a \dashv b$ and $a \vdash b$ satisfying
these polynomial identities:
  \allowdisplaybreaks
  \begin{alignat*}{3}
  ( a \dashv b ) \dashv c &= a \dashv ( b \dashv c ),
  &\quad
  ( a \vdash b ) \vdash c &= a \vdash ( b \vdash c ),
  &\quad
  ( a \vdash b ) \dashv c &= a \vdash ( b \dashv c ),
  \\
  ( a \vdash b ) \vdash c &= ( a \dashv b ) \vdash c,
  &\quad
  a \dashv ( b \dashv c ) &= a \dashv ( b \vdash c ).
  \end{alignat*}
Let $w = \overline{\, a_1 \cdots a_n}$ ($a_1, \hdots, a_n \in X$) be a dialgebra
monomial over $X$; the bar indicates some placement of parentheses and choice
of operations. If $w \in X$ then the \textbf{center} $c(w) = w$; otherwise
$c( w_1 \dashv w_2 ) = c(w_1)$ and $c( w_1 \vdash w_2 ) = c(w_2)$.

\begin{lemma} \label{dimonoidcalculus}
\cite{Loday2} If $w = \overline{\, a_1 \cdots a_n}$ and $c(w) = a_k$ then
  \[
  w = ( a_1 \vdash \cdots \vdash a_{k-1} ) \vdash a_k \dashv ( a_{k+1} \dashv \cdots \dashv a_n ).
  \]
\end{lemma}

\noindent We write $w = a_1 \cdots a_{k-1} \widehat a_k a_{k+1} \cdots a_n$ for
this \textbf{normal form} of $w$.

\begin{lemma} \label{freedialgebra}
\cite{Loday2} The monomials $a_1 \cdots a_{k-1} \widehat a_k a_{k+1} \cdots a_n$
($1 \le k \le n$) for $a_1, \hdots, a_n \in X$ form a basis of the free dialgebra on $X$.
\end{lemma}

Vel\'asquez and Felipe \cite{VelasquezFelipe1, VelasquezFelipe2} introduced the
\textbf{(right) quasi-Jordan product} in a dialgebra over a field of
characteristic $\ne 2$:
 \[
 a \triangleleft b = \tfrac12 ( a \dashv b \,+\, b \vdash a ).
 \]
We omit the operation symbol and the scalar $\tfrac12$, and write $ab = a \dashv b
\,+\, b \vdash a$.

\begin{lemma} \label{rightcommutative}
\cite{VelasquezFelipe1} The quasi-Jordan product satisfies the right
commutative identity:
  \[
  a ( b c ) = a ( c b ).
  \]
\end{lemma}

\noindent Expanding the left side gives a result which is symmetric in $b$ and
$c$:
  \allowdisplaybreaks
  \begin{align*}
  &
  a ( b c )
  =
  a \dashv ( b \dashv c + c \vdash b ) + ( b \dashv c + c \vdash b ) \vdash a
  \\
  &=
  a \dashv ( b \dashv c ) + a \dashv ( c \vdash b ) + ( b \dashv c ) \vdash a + ( c \vdash b ) \vdash a
  =
  \widehat a b c + \widehat a c b + b c \widehat a + c b \widehat a.
  \end{align*}

\begin{lemma} \label{rightquasijordan}
\cite{VelasquezFelipe1} The quasi-Jordan product satisfies the right
quasi-Jordan identity:
  \[
  ( b a ) a^2 = ( b a^2 ) a.
  \]
\end{lemma}

\noindent Expanding both sides gives the same result:
  \allowdisplaybreaks
  \begin{align*}
  &
  ( b a ) a^2
  =
  ( b \dashv a + a \vdash b ) \dashv ( a \dashv a + a \vdash a )
  +
  ( a \dashv a + a \vdash a ) \vdash ( b \dashv a + a \vdash b )
  \\
  &=
  ( b \dashv a ) \dashv ( a \dashv a )
  +
  ( b \dashv a ) \dashv ( a \vdash a )
  +
  ( a \vdash b ) \dashv ( a \dashv a )
  +
  ( a \vdash b ) \dashv ( a \vdash a )
  \\
  &\quad
  +
  ( a \dashv a ) \vdash ( b \dashv a )
  +
  ( a \dashv a ) \vdash ( a \vdash b )
  +
  ( a \vdash a ) \vdash ( b \dashv a )
  +
  ( a \vdash a ) \vdash ( a \vdash b )
  \\
  &=
  2 \widehat b a a a
  +
  2 a \widehat b a a
  +
  2 a a \widehat b a
  +
  2 a a a \widehat b
  \\
  &=
  ( b \dashv ( a \dashv a ) ) \dashv a
  +
  ( b \dashv ( a \vdash a ) ) \dashv a
  +
  ( ( a \dashv a ) \vdash b ) \dashv a
  +
  ( ( a \vdash a ) \vdash b ) \dashv a
  \\
  &\quad
  +
  a \vdash ( b \dashv ( a \dashv a ) )
  +
  a \vdash ( b \dashv ( a \vdash a ) )
  +
  a \vdash ( ( a \dashv a ) \vdash b )
  +
  a \vdash ( ( a \vdash a ) \vdash b )
  \\
  &=
  ( b \dashv ( a \dashv a + a \vdash a ) + ( a \dashv a + a \vdash a ) \vdash b ) \dashv a
  \\
  &\quad
  +
  a \vdash ( b \dashv ( a \dashv a + a \vdash a ) + ( a \dashv a + a \vdash a ) \vdash b )
  \\
  &=
  ( b a^2 ) a.
  \end{align*}

\begin{remark}
The (right) quasi-Jordan product does not satisfy $a^2 ( a b ) = a ( a^2 b )$:
  \allowdisplaybreaks
  \begin{align*}
  &
  a^2 ( a b )
  =
  ( a \dashv a + a \vdash a ) \dashv ( a \dashv b + b \vdash a )
  +
  ( a \dashv b + b \vdash a ) \vdash ( a \dashv a + a \vdash a )
  \\
  &=
  ( a \dashv a ) \dashv ( a \dashv b )
  +
  ( a \dashv a ) \dashv ( b \vdash a )
  +
  ( a \vdash a ) \dashv ( a \dashv b )
  +
  ( a \vdash a ) \dashv ( b \vdash a )
  \\
  &\quad
  +
  ( a \dashv b ) \vdash ( a \dashv a )
  +
  ( a \dashv b ) \vdash ( a \vdash a )
  +
  ( b \vdash a ) \vdash ( a \dashv a )
  +
  ( b \vdash a ) \vdash ( a \vdash a )
  \\
  &=
  \widehat a a a b
  +
  \widehat a a b a
  +
  a \widehat a a b
  +
  a \widehat a b a
  +
  a b \widehat a a
  +
  a b a \widehat a
  +
  b a \widehat a a
  +
  b a a \widehat a,
  \\
  &
  a ( a^2 b )
  =
  a \dashv ( ( a \dashv a + a \vdash a ) \dashv b + b \vdash ( a \dashv a + a \vdash a ) )
  \\
  &\quad
  +
  ( ( a \dashv a + a \vdash a ) \dashv b + b \vdash ( a \dashv a + a \vdash a ) ) \vdash a
  \\
  &=
  a \dashv ( ( a \dashv a ) \dashv b )
  +
  a \dashv ( ( a \vdash a ) \dashv b )
  +
  a \dashv ( b \vdash ( a \dashv a ) )
  +
  a \dashv ( b \vdash ( a \vdash a ) )
  \\
  &\quad
  +
  ( ( a \dashv a ) \dashv b ) \vdash a
  +
  ( ( a \vdash a ) \dashv b ) \vdash a
  +
  ( b \vdash ( a \dashv a ) ) \vdash a
  +
  ( b \vdash ( a \vdash a ) ) \vdash a
  \\
  &=
  2 \widehat a a a b
  +
  2 \widehat a b a a
  +
  2 a a b \widehat a
  +
  2 b a a \widehat a.
  \end{align*}
\end{remark}

In this paper we determine a set of generators for the identities of degree $\le 4$
satisfied by the quasi-Jordan product over a field of characteristic 0. We may assume
that the identities are multilinear by Zhevlakov et al.~\cite{Zhevlakov}, Chapter 1. We
use the computer algebra system Maple \cite{Maple} for calculations with large
matrices.

\begin{theorem} \label{degree3theorem}
Every polynomial identity of degree $\le 3$ satisfied by the quasi-Jordan
product is a consequence of right commutativity.
\end{theorem}

\begin{proof}
The two association types $(\ast\ast)\ast$, $\ast(\ast\ast)$ and six permutations
of $a,b,c$ give 12 nonassociative monomials, ordered as follows:
  \[
  (ab)c,
  (ac)b,
  (ba)c,
  (bc)a,
  (ca)b,
  (cb)a,
  a(bc),
  a(cb),
  b(ac),
  b(ca),
  c(ab),
  c(ba).
  \]
There are 18 dialgebra monomials, ordered as follows:
  \[
  \widehat abc,
  \widehat acb,
  \widehat bac,
  \widehat bca,
  \widehat cab,
  \widehat cba,
  a \widehat bc,
  a \widehat cb,
  b \widehat ac,
  b \widehat ca,
  c \widehat ab,
  c \widehat ba,
  ab \widehat c,
  ac \widehat b,
  ba \widehat c,
  bc \widehat a,
  ca \widehat b,
  cb \widehat a.
  \]
We expand each nonassociative monomial with the quasi-Jordan product:
  \allowdisplaybreaks
  \begin{align*}
  (ab)c
  &=
  ( a \dashv b + b \vdash a ) \dashv c + c \vdash ( a \dashv b + b \vdash a )
  =
  \widehat abc + b \widehat ac + c \widehat ab + cb \widehat a,
  \\
  a(bc)
  &=
  a \dashv ( b \dashv c + c \vdash b ) + ( b \dashv c + c \vdash b ) \vdash a
  =
  \widehat abc + \widehat acb + bc \widehat a + cb \widehat a.
  \end{align*}
By permutation of $a,b,c$ we obtain the remaining expansions:
  \allowdisplaybreaks
  \begin{alignat*}{2}
  ( a c ) b &= \widehat acb + c \widehat ab + b \widehat ac + bc \widehat a,
  &\qquad
  a ( c b ) &= \widehat acb + \widehat abc + cb \widehat a + bc \widehat a,
  \\
  ( b a ) c &= \widehat bac + a \widehat bc + c \widehat ba + ca \widehat b,
  &\qquad
  b ( a c ) &= \widehat bac + \widehat bca + ac \widehat b + ca \widehat b,
  \\
  ( b c ) a &= \widehat bca + c \widehat ba + a \widehat bc + ac \widehat b,
  &\qquad
  b ( c a ) &= \widehat bca + \widehat bac + ca \widehat b + ac \widehat b,
  \\
  ( c a ) b &= \widehat cab + a \widehat cb + b \widehat ca + ba \widehat c,
  &\qquad
  c ( a b ) &= \widehat cab + \widehat cba + ab \widehat c + ba \widehat c,
  \\
  ( c b ) a &= \widehat cba + b \widehat ca + a \widehat cb + ab \widehat c,
  &\qquad
  c ( b a ) &= \widehat cba + \widehat cab + ba \widehat c + ab \widehat c.
  \end{alignat*}
Let $E_3$ be the $18 \times 12$ matrix whose $(i,j)$ entry is the coefficient
of the $i$-th dialgebra monomial in the expansion of the $j$-th nonassociative
monomial. The nullspace of $E_3$ consists of the identities in degree 3
satisfied by the quasi-Jordan product. The matrix $E_3$ is on the left side of
Table \ref{degree3matrix} ($.$ for $0$); on the right side is the row canonical
form (upper block) and the canonical basis of the nullspace (lower block). The
nullspace basis vectors are the coefficient vectors of the permutations of the
right commutative identity: $- a ( b c ) + a ( c b )$, $- b ( a c ) + b ( c a
)$, $- c ( a b ) + c ( b a )$.
\end{proof}

  \begin{table}[t] \tiny
  \[
  \left[
  \begin{array}{cccccccccccc}
  1 &\!\! . &\!\! . &\!\! . &\!\! . &\!\! . &\!\! 1 &\!\! 1 &\!\! . &\!\! . &\!\! . &\!\! . \\
  . &\!\! 1 &\!\! . &\!\! . &\!\! . &\!\! . &\!\! 1 &\!\! 1 &\!\! . &\!\! . &\!\! . &\!\! . \\
  . &\!\! . &\!\! 1 &\!\! . &\!\! . &\!\! . &\!\! . &\!\! . &\!\! 1 &\!\! 1 &\!\! . &\!\! . \\
  . &\!\! . &\!\! . &\!\! 1 &\!\! . &\!\! . &\!\! . &\!\! . &\!\! 1 &\!\! 1 &\!\! . &\!\! . \\
  . &\!\! . &\!\! . &\!\! . &\!\! 1 &\!\! . &\!\! . &\!\! . &\!\! . &\!\! . &\!\! 1 &\!\! 1 \\
  . &\!\! . &\!\! . &\!\! . &\!\! . &\!\! 1 &\!\! . &\!\! . &\!\! . &\!\! . &\!\! 1 &\!\! 1 \\
  . &\!\! . &\!\! 1 &\!\! 1 &\!\! . &\!\! . &\!\! . &\!\! . &\!\! . &\!\! . &\!\! . &\!\! . \\
  . &\!\! . &\!\! . &\!\! . &\!\! 1 &\!\! 1 &\!\! . &\!\! . &\!\! . &\!\! . &\!\! . &\!\! . \\
  1 &\!\! 1 &\!\! . &\!\! . &\!\! . &\!\! . &\!\! . &\!\! . &\!\! . &\!\! . &\!\! . &\!\! . \\
  . &\!\! . &\!\! . &\!\! . &\!\! 1 &\!\! 1 &\!\! . &\!\! . &\!\! . &\!\! . &\!\! . &\!\! . \\
  1 &\!\! 1 &\!\! . &\!\! . &\!\! . &\!\! . &\!\! . &\!\! . &\!\! . &\!\! . &\!\! . &\!\! . \\
  . &\!\! . &\!\! 1 &\!\! 1 &\!\! . &\!\! . &\!\! . &\!\! . &\!\! . &\!\! . &\!\! . &\!\! . \\
  . &\!\! . &\!\! . &\!\! . &\!\! . &\!\! 1 &\!\! . &\!\! . &\!\! . &\!\! . &\!\! 1 &\!\! 1 \\
  . &\!\! . &\!\! . &\!\! 1 &\!\! . &\!\! . &\!\! . &\!\! . &\!\! 1 &\!\! 1 &\!\! . &\!\! . \\
  . &\!\! . &\!\! . &\!\! . &\!\! 1 &\!\! . &\!\! . &\!\! . &\!\! . &\!\! . &\!\! 1 &\!\! 1 \\
  . &\!\! 1 &\!\! . &\!\! . &\!\! . &\!\! . &\!\! 1 &\!\! 1 &\!\! . &\!\! . &\!\! . &\!\! . \\
  . &\!\! . &\!\! 1 &\!\! . &\!\! . &\!\! . &\!\! . &\!\! . &\!\! 1 &\!\! 1 &\!\! . &\!\! . \\
  1 &\!\! . &\!\! . &\!\! . &\!\! . &\!\! . &\!\! 1 &\!\! 1 &\!\! . &\!\! . &\!\! . &\!\! .
  \end{array}
  \right]
  \quad
  \left[
  \begin{array}{cccccccccccc}
  1 &\!\! . &\!\! . &\!\! . &\!\! . &\!\! . &\!\! . &\!\! . &\!\! . &\!\! . &\!\! . &\!\! . \\
  . &\!\! 1 &\!\! . &\!\! . &\!\! . &\!\! . &\!\! . &\!\! . &\!\! . &\!\! . &\!\! . &\!\! . \\
  . &\!\! . &\!\! 1 &\!\! . &\!\! . &\!\! . &\!\! . &\!\! . &\!\! . &\!\! . &\!\! . &\!\! . \\
  . &\!\! . &\!\! . &\!\! 1 &\!\! . &\!\! . &\!\! . &\!\! . &\!\! . &\!\! . &\!\! . &\!\! . \\
  . &\!\! . &\!\! . &\!\! . &\!\! 1 &\!\! . &\!\! . &\!\! . &\!\! . &\!\! . &\!\! . &\!\! . \\
  . &\!\! . &\!\! . &\!\! . &\!\! . &\!\! 1 &\!\! . &\!\! . &\!\! . &\!\! . &\!\! . &\!\! . \\
  . &\!\! . &\!\! . &\!\! . &\!\! . &\!\! . &\!\! 1 &\!\! 1 &\!\! . &\!\! . &\!\! . &\!\! . \\
  . &\!\! . &\!\! . &\!\! . &\!\! . &\!\! . &\!\! . &\!\! . &\!\! 1 &\!\! 1 &\!\! . &\!\! . \\
  . &\!\! . &\!\! . &\!\! . &\!\! . &\!\! . &\!\! . &\!\! . &\!\! . &\!\! . &\!\! 1 &\!\! 1 \\
  \hline
  . &\!\! . &\!\! . &\!\! . &\!\! . &\!\! . &\!\! -1 &\!\! 1 &\!\!  . &\!\! . &\!\!  . &\!\! . \\
  . &\!\! . &\!\! . &\!\! . &\!\! . &\!\! . &\!\!  . &\!\! . &\!\! -1 &\!\! 1 &\!\!  . &\!\! . \\
  . &\!\! . &\!\! . &\!\! . &\!\! . &\!\! . &\!\!  . &\!\! . &\!\!  . &\!\! . &\!\! -1 &\!\! 1
  \end{array}
  \right]
  \]
  \caption{The matrix $E_3$, its row canonical form and nullspace basis}
  \label{degree3matrix}
  \end{table}

\begin{theorem} \label{degree4theorem}
Every polynomial identity of degree $\le 4$ satisfied by the quasi-Jordan
product is a consequence of right commutativity and these identities in degree
4:
  \allowdisplaybreaks
  \begin{align*}
  J &=
  ( a ( b c ) ) d
  +
  ( a ( b d ) ) c
  +
  ( a ( c d ) ) b
  -
  ( a b ) ( c d )
  -
  ( a c ) ( b d )
  -
  ( a d ) ( b c ),
  \\
  K &=
  ( ( a b ) d ) c
  +
  ( ( a c ) d ) b
  -
  ( a ( b c ) ) d
  -
  ( a ( b d ) ) c
  -
  ( a ( c d ) ) b
  +
  a ( ( b c ) d ).
  \end{align*}
The identities $J$ and $K$ are independent: neither implies the other.
\end{theorem}

\begin{proof}
The 5 association types are $(({\ast}{\ast}){\ast}){\ast}$,
$({\ast}({\ast}{\ast})){\ast}$, $({\ast}{\ast})({\ast}{\ast})$,
${\ast}(({\ast}{\ast}){\ast})$, ${\ast}({\ast}({\ast}{\ast}))$. Right
commutativity eliminates type 5, $a(b(cd)) = a((cd)b)$, and implies symmetries
in types 1--4: $((ab)c)d$ has no symmetry, but $(a(bc))d = (a(cb))d$, $(ab)(cd)
= (ab)(dc)$, and $a((bc)d) = a((cb)d)$ since $a((bc)d) = a(d(bc)) = a(d(cb)) =
a((cb)d)$. Each symmetry divides the number of monomials by 2, so types 1, 2,
3, 4 have 24, 12, 12, 12 monomials. We order these 60 monomials by association
type and by lex order of the permutation (Table \ref{nonassociativemonomials}).
There are 96 dialgebra monomials: in each permutation each factor can be the
center. We order these monomials by the position of the center and by lex order
of the permutation (Table \ref{dialgebramonomials}). We expand each
nonassociative monomial and replace each dialgebra term by its normal form:
  \allowdisplaybreaks
  \begin{align*}
  &
  ((ab)c)d
  \\
  &=
  ( ( a \dashv b ) \dashv c ) \dashv d
  +
  d \vdash ( ( a \dashv b ) \dashv c )
  +
  ( c \vdash ( a \dashv b ) ) \dashv d
  +
  d \vdash ( c \vdash ( a \dashv b ) )
  \\
  &\quad
  +
  ( ( b \vdash a ) \dashv c ) \dashv d
  +
  d \vdash ( ( b \vdash a ) \dashv c )
  +
  ( c \vdash ( b \vdash a ) ) \dashv d
  +
  d \vdash ( c \vdash ( b \vdash a ) )
  \\
  &=
  \widehat{a} b c d
  +
  d \widehat{a} b c
  +
  c \widehat{a} b d
  +
  d c \widehat{a} b
  +
  b \widehat{a} c d
  +
  d b \widehat{a} c
  +
  c b \widehat{a} d
  +
  d c b \widehat{a},
  \\
  &
  (a(bc))d
  \\
  &=
  ( a \dashv ( b \dashv c ) ) \dashv d
  +
  d \vdash ( a \dashv ( b \dashv c ) )
  +
  ( ( b \dashv c ) \vdash a ) \dashv d
  +
  d \vdash ( ( b \dashv c ) \vdash a )
  \\
  &\quad
  +
  ( a \dashv ( c \vdash b ) ) \dashv d
  +
  d \vdash ( a \dashv ( c \vdash b ) )
  +
  ( ( c \vdash b ) \vdash a ) \dashv d
  +
  d \vdash ( ( c \vdash b ) \vdash a )
  \\
  &=
  \widehat{a} b c d
  +
  d \widehat{a} b c
  +
  b c \widehat{a} d
  +
  d b c \widehat{a}
  +
  \widehat{a} c b d
  +
  d \widehat{a} c b
  +
  c b \widehat{a} d
  +
  d c b \widehat{a},
  \\
  &
  (ab)(cd)
  \\
  &=
  ( a \dashv b ) \dashv ( c \dashv d )
  +
  ( c \dashv d ) \vdash ( a \dashv b )
  +
  ( a \dashv b ) \dashv ( d \vdash c )
  +
  ( d \vdash c ) \vdash ( a \dashv b )
  \\
  &\quad
  +
  ( b \vdash a ) \dashv ( c \dashv d )
  +
  ( c \dashv d ) \vdash ( b \vdash a )
  +
  ( b \vdash a ) \dashv ( d \vdash c )
  +
  ( d \vdash c ) \vdash ( b \vdash a )
  \\
  &=
  \widehat{a} b c d
  +
  c d \widehat{a} b
  +
  \widehat{a} b d c
  +
  d c \widehat{a} b
  +
  b \widehat{a} c d
  +
  c d b \widehat{a}
  +
  b \widehat{a} d c
  +
  d c b \widehat{a},
  \\
  &
  a((bc)d)
  \\
  &=
  a \dashv ( ( b \dashv c ) \dashv d )
  +
  ( ( b \dashv c ) \dashv d ) \vdash a
  +
  a \dashv ( d \vdash ( b \dashv c ) )
  +
  ( d \vdash ( b \dashv c ) ) \vdash a
  \\
  &\quad
  +
  a \dashv ( ( c \vdash b ) \dashv d )
  +
  ( ( c \vdash b ) \dashv d ) \vdash a
  +
  a \dashv ( d \vdash ( c \vdash b ) )
  +
  ( d \vdash ( c \vdash b ) ) \vdash a
  \\
  &=
  \widehat{a} b c d
  +
  b c d \widehat{a}
  +
  \widehat{a} d b c
  +
  d b c \widehat{a}
  +
  \widehat{a} c b d
  +
  c b d \widehat{a}
  +
  \widehat{a} d c b
  +
  d c b \widehat{a}.
  \end{align*}
Because of its large size, we omit the $96 \times 60$ matrix $E_4$ whose $(i,j)$ entry
is the coefficient of the $i$-th dialgebra monomial in the expansion of the $j$-th
nonassociative monomial. Table \ref{degree4expansionrcfnull} gives the row canonical
form of $E_4$ in the upper block ($+$ for 1, $-$ for $-1$), and the canonical basis of
the nullspace of $E_4$ in the lower block.
The rank is 44 and the nullity is 16. Nullspace basis vectors 1--4 represent
the first 4 identities in Table \ref{nullspacebasis} which span an
$S_4$-module generated by the identity $J$. In $J$, if we (simultaneously)
replace $a$ by $b$, and $b,c,d$ by $a$, then we obtain $3 ( b a^2 ) a - 3 ( b a
) a^2$; this is equivalent to the right quasi-Jordan identity (since
$\mathrm{char}\,\mathbb{F} \ne 3$). If we linearize the right quasi-Jordan
identity and apply right commutativity then we obtain
  \allowdisplaybreaks
  \begin{align*}
  &
  ( a ( b c ) ) d
  +
  ( a ( b d ) ) c
  +
  ( a ( c b ) ) d
  +
  ( a ( c d ) ) b
  +
  ( a ( d b ) ) c
  +
  ( a ( d c ) ) b
  \\
  &\quad
  -
  ( a b ) ( c d )
  -
  ( a b ) ( d c )
  -
  ( a c ) ( b d )
  -
  ( a c ) ( d b )
  -
  ( a d ) ( b c )
  -
  ( a d ) ( c b )
  \\
  &=
  2 ( a ( b c ) ) d
  +
  2 ( a ( b d ) ) c
  +
  2 ( a ( c d ) ) b
  -
  2 ( a b ) ( c d )
  -
  2 ( a c ) ( b d )
  -
  2 ( a d ) ( b c );
  \end{align*}
this is equivalent to identity $J$ (since $\mathrm{char}\,\mathbb{F} \ne 2$).
Nullspace basis vectors 5--16 represent the last 12 identities in Table
\ref{nullspacebasis} which span an $S_4$-module generated by the identity $K$.
The $S_4$-submodules generated by $J$ and $K$ have zero intersection, and so
$J$ and $K$ are independent.
\end{proof}

  \begin{table} \small
  {
  \[
  \begin{array}{llllllll}
  ((ab)c)d & ((ab)d)c & ((ac)b)d & ((ac)d)b &
  ((ad)b)c & ((ad)c)b & ((ba)c)d & ((ba)d)c \\
  ((bc)a)d & ((bc)d)a & ((bd)a)c & ((bd)c)a &
  ((ca)b)d & ((ca)d)b & ((cb)a)d & ((cb)d)a \\
  ((cd)a)b & ((cd)b)a & ((da)b)c & ((da)c)b &
  ((db)a)c & ((db)c)a & ((dc)a)b & ((dc)b)a \\
  (a(bc))d & (a(bd))c & (a(cd))b & (b(ac))d &
  (b(ad))c & (b(cd))a & (c(ab))d & (c(ad))b \\
  (c(bd))a & (d(ab))c & (d(ac))b & (d(bc))a &
  (ab)(cd) & (ac)(bd) & (ad)(bc) & (ba)(cd) \\
  (bc)(ad) & (bd)(ac) & (ca)(bd) & (cb)(ad) &
  (cd)(ab) & (da)(bc) & (db)(ac) & (dc)(ab) \\
  a((bc)d) & a((bd)c) & a((cd)b) & b((ac)d) &
  b((ad)c) & b((cd)a) & c((ab)d) & c((ad)b) \\
  c((bd)a) & d((ab)c) & d((ac)b) & d((bc)a).
  \end{array}
  \]
  \caption{Nonassociative monomials in degree 4}
  \label{nonassociativemonomials}
  }
  {
  \[
  \begin{array}{llllllllllll}
  \\
  \widehat{a}bcd & \widehat{a}bdc & \widehat{a}cbd &
  \widehat{a}cdb & \widehat{a}dbc & \widehat{a}dcb &
  \widehat{b}acd & \widehat{b}adc & \widehat{b}cad &
  \widehat{b}cda & \widehat{b}dac & \widehat{b}dca \\
  \widehat{c}abd & \widehat{c}adb & \widehat{c}bad &
  \widehat{c}bda & \widehat{c}dab & \widehat{c}dba &
  \widehat{d}abc & \widehat{d}acb & \widehat{d}bac &
  \widehat{d}bca & \widehat{d}cab & \widehat{d}cba \\
  a\widehat{b}cd & a\widehat{b}dc & a\widehat{c}bd &
  a\widehat{c}db & a\widehat{d}bc & a\widehat{d}cb &
  b\widehat{a}cd & b\widehat{a}dc & b\widehat{c}ad &
  b\widehat{c}da & b\widehat{d}ac & b\widehat{d}ca \\
  c\widehat{a}bd & c\widehat{a}db & c\widehat{b}ad &
  c\widehat{b}da & c\widehat{d}ab & c\widehat{d}ba &
  d\widehat{a}bc & d\widehat{a}cb & d\widehat{b}ac &
  d\widehat{b}ca & d\widehat{c}ab & d\widehat{c}ba \\
  ab\widehat{c}d & ab\widehat{d}c & ac\widehat{b}d &
  ac\widehat{d}b & ad\widehat{b}c & ad\widehat{c}b &
  ba\widehat{c}d & ba\widehat{d}c & bc\widehat{a}d &
  bc\widehat{d}a & bd\widehat{a}c & bd\widehat{c}a \\
  ca\widehat{b}d & ca\widehat{d}b & cb\widehat{a}d &
  cb\widehat{d}a & cd\widehat{a}b & cd\widehat{b}a &
  da\widehat{b}c & da\widehat{c}b & db\widehat{a}c &
  db\widehat{c}a & dc\widehat{a}b & dc\widehat{b}a \\
  abc\widehat{d} & abd\widehat{c} & acb\widehat{d} &
  acd\widehat{b} & adb\widehat{c} & adc\widehat{b} &
  bac\widehat{d} & bad\widehat{c} & bca\widehat{d} &
  bcd\widehat{a} & bda\widehat{c} & bdc\widehat{a} \\
  cab\widehat{d} & cad\widehat{b} & cba\widehat{d} &
  cbd\widehat{a} & cda\widehat{b} & cdb\widehat{a} &
  dab\widehat{c} & dac\widehat{b} & dba\widehat{c} &
  dbc\widehat{a} & dca\widehat{b} & dcb\widehat{a}.
  \end{array}
  \]
  \caption{Dialgebra monomials in degree 4}
  \label{dialgebramonomials}
  }
  \end{table}

  \begin{table}
  \begin{center} \tiny
  \begin{alignat*}{60}
  &  + & &\,. & &\,. & &\,. & &\,. & &\,. & &\,. & &\,. & &\,. & &\,. & &\,. & &\,. &
  &\,. & &\,. & &\,. & &\,. & &\,. & &\,. & &\,. & &\,. & &\,. & &\,. & &\,. & &\,. &
  &\,. & &\,. & &\,. & &\,. & &\,. & &\,. & &\,. & &\,. & &\,. & &\,. & &\,. & &\,. &
  &\,. & &\,. & &\,. & &\,. & &\,. & &\,. & &\,. & &\,. & &\,. & &\,. & &\,. & &\,. &
  &\,. & &  - & &\,. & &\,. & &\,. & &\,. & &\,. & &\,. & &\,. & &\,. & &\,. & &\,.
  \\[-2pt] 
  &\,. & &  + & &\,. & &\,. & &\,. & &\,. & &\,. & &\,. & &\,. & &\,. & &\,. & &\,. &
  &\,. & &\,. & &\,. & &\,. & &\,. & &\,. & &\,. & &\,. & &\,. & &\,. & &\,. & &\,. &
  &\,. & &\,. & &\,. & &\,. & &\,. & &\,. & &\,. & &\,. & &\,. & &\,. & &\,. & &\,. &
  &\,. & &\,. & &\,. & &\,. & &\,. & &\,. & &\,. & &\,. & &\,. & &\,. & &\,. & &\,. &
  &  - & &\,. & &\,. & &\,. & &\,. & &\,. & &\,. & &\,. & &\,. & &\,. & &\,. & &\,.
  \\[-2pt] 
  &\,. & &\,. & &  + & &\,. & &\,. & &\,. & &\,. & &\,. & &\,. & &\,. & &\,. & &\,. &
  &\,. & &\,. & &\,. & &\,. & &\,. & &\,. & &\,. & &\,. & &\,. & &\,. & &\,. & &\,. &
  &\,. & &\,. & &\,. & &\,. & &\,. & &\,. & &\,. & &\,. & &\,. & &\,. & &\,. & &\,. &
  &\,. & &\,. & &\,. & &\,. & &\,. & &\,. & &\,. & &\,. & &\,. & &\,. & &\,. & &\,. &
  &\,. & &\,. & &  - & &\,. & &\,. & &\,. & &\,. & &\,. & &\,. & &\,. & &\,. & &\,.
  \\[-2pt] 
  &\,. & &\,. & &\,. & &  + & &\,. & &\,. & &\,. & &\,. & &\,. & &\,. & &\,. & &\,. &
  &\,. & &\,. & &\,. & &\,. & &\,. & &\,. & &\,. & &\,. & &\,. & &\,. & &\,. & &\,. &
  &\,. & &\,. & &\,. & &\,. & &\,. & &\,. & &\,. & &\,. & &\,. & &\,. & &\,. & &\,. &
  &\,. & &\,. & &\,. & &\,. & &\,. & &\,. & &\,. & &\,. & &\,. & &\,. & &\,. & &\,. &
  &  - & &\,. & &\,. & &\,. & &\,. & &\,. & &\,. & &\,. & &\,. & &\,. & &\,. & &\,.
  \\[-2pt] 
  &\,. & &\,. & &\,. & &\,. & &  + & &\,. & &\,. & &\,. & &\,. & &\,. & &\,. & &\,. &
  &\,. & &\,. & &\,. & &\,. & &\,. & &\,. & &\,. & &\,. & &\,. & &\,. & &\,. & &\,. &
  &\,. & &\,. & &\,. & &\,. & &\,. & &\,. & &\,. & &\,. & &\,. & &\,. & &\,. & &\,. &
  &\,. & &\,. & &\,. & &\,. & &\,. & &\,. & &\,. & &\,. & &\,. & &\,. & &\,. & &\,. &
  &\,. & &\,. & &  - & &\,. & &\,. & &\,. & &\,. & &\,. & &\,. & &\,. & &\,. & &\,.
  \\[-2pt] 
  &\,. & &\,. & &\,. & &\,. & &\,. & &  + & &\,. & &\,. & &\,. & &\,. & &\,. & &\,. &
  &\,. & &\,. & &\,. & &\,. & &\,. & &\,. & &\,. & &\,. & &\,. & &\,. & &\,. & &\,. &
  &\,. & &\,. & &\,. & &\,. & &\,. & &\,. & &\,. & &\,. & &\,. & &\,. & &\,. & &\,. &
  &\,. & &\,. & &\,. & &\,. & &\,. & &\,. & &\,. & &\,. & &\,. & &\,. & &\,. & &\,. &
  &\,. & &  - & &\,. & &\,. & &\,. & &\,. & &\,. & &\,. & &\,. & &\,. & &\,. & &\,.
  \\[-2pt] 
  &\,. & &\,. & &\,. & &\,. & &\,. & &\,. & &  + & &\,. & &\,. & &\,. & &\,. & &\,. &
  &\,. & &\,. & &\,. & &\,. & &\,. & &\,. & &\,. & &\,. & &\,. & &\,. & &\,. & &\,. &
  &\,. & &\,. & &\,. & &\,. & &\,. & &\,. & &\,. & &\,. & &\,. & &\,. & &\,. & &\,. &
  &\,. & &\,. & &\,. & &\,. & &\,. & &\,. & &\,. & &\,. & &\,. & &\,. & &\,. & &\,. &
  &\,. & &\,. & &\,. & &\,. & &  - & &\,. & &\,. & &\,. & &\,. & &\,. & &\,. & &\,.
  \\[-2pt] 
  &\,. & &\,. & &\,. & &\,. & &\,. & &\,. & &\,. & &  + & &\,. & &\,. & &\,. & &\,. &
  &\,. & &\,. & &\,. & &\,. & &\,. & &\,. & &\,. & &\,. & &\,. & &\,. & &\,. & &\,. &
  &\,. & &\,. & &\,. & &\,. & &\,. & &\,. & &\,. & &\,. & &\,. & &\,. & &\,. & &\,. &
  &\,. & &\,. & &\,. & &\,. & &\,. & &\,. & &\,. & &\,. & &\,. & &\,. & &\,. & &\,. &
  &\,. & &\,. & &\,. & &  - & &\,. & &\,. & &\,. & &\,. & &\,. & &\,. & &\,. & &\,.
  \\[-2pt] 
  &\,. & &\,. & &\,. & &\,. & &\,. & &\,. & &\,. & &\,. & &  + & &\,. & &\,. & &\,. &
  &\,. & &\,. & &\,. & &\,. & &\,. & &\,. & &\,. & &\,. & &\,. & &\,. & &\,. & &\,. &
  &\,. & &\,. & &\,. & &\,. & &\,. & &\,. & &\,. & &\,. & &\,. & &\,. & &\,. & &\,. &
  &\,. & &\,. & &\,. & &\,. & &\,. & &\,. & &\,. & &\,. & &\,. & &\,. & &\,. & &\,. &
  &\,. & &\,. & &\,. & &\,. & &\,. & &  - & &\,. & &\,. & &\,. & &\,. & &\,. & &\,.
  \\[-2pt] 
  &\,. & &\,. & &\,. & &\,. & &\,. & &\,. & &\,. & &\,. & &\,. & &  + & &\,. & &\,. &
  &\,. & &\,. & &\,. & &\,. & &\,. & &\,. & &\,. & &\,. & &\,. & &\,. & &\,. & &\,. &
  &\,. & &\,. & &\,. & &\,. & &\,. & &\,. & &\,. & &\,. & &\,. & &\,. & &\,. & &\,. &
  &\,. & &\,. & &\,. & &\,. & &\,. & &\,. & &\,. & &\,. & &\,. & &\,. & &\,. & &\,. &
  &\,. & &\,. & &\,. & &  - & &\,. & &\,. & &\,. & &\,. & &\,. & &\,. & &\,. & &\,.
  \\[-2pt] 
  &\,. & &\,. & &\,. & &\,. & &\,. & &\,. & &\,. & &\,. & &\,. & &\,. & &  + & &\,. &
  &\,. & &\,. & &\,. & &\,. & &\,. & &\,. & &\,. & &\,. & &\,. & &\,. & &\,. & &\,. &
  &\,. & &\,. & &\,. & &\,. & &\,. & &\,. & &\,. & &\,. & &\,. & &\,. & &\,. & &\,. &
  &\,. & &\,. & &\,. & &\,. & &\,. & &\,. & &\,. & &\,. & &\,. & &\,. & &\,. & &\,. &
  &\,. & &\,. & &\,. & &\,. & &\,. & &  - & &\,. & &\,. & &\,. & &\,. & &\,. & &\,.
  \\[-2pt] 
  &\,. & &\,. & &\,. & &\,. & &\,. & &\,. & &\,. & &\,. & &\,. & &\,. & &\,. & &  + &
  &\,. & &\,. & &\,. & &\,. & &\,. & &\,. & &\,. & &\,. & &\,. & &\,. & &\,. & &\,. &
  &\,. & &\,. & &\,. & &\,. & &\,. & &\,. & &\,. & &\,. & &\,. & &\,. & &\,. & &\,. &
  &\,. & &\,. & &\,. & &\,. & &\,. & &\,. & &\,. & &\,. & &\,. & &\,. & &\,. & &\,. &
  &\,. & &\,. & &\,. & &\,. & &  - & &\,. & &\,. & &\,. & &\,. & &\,. & &\,. & &\,.
  \\[-2pt] 
  &\,. & &\,. & &\,. & &\,. & &\,. & &\,. & &\,. & &\,. & &\,. & &\,. & &\,. & &\,. &
  &  + & &\,. & &\,. & &\,. & &\,. & &\,. & &\,. & &\,. & &\,. & &\,. & &\,. & &\,. &
  &\,. & &\,. & &\,. & &\,. & &\,. & &\,. & &\,. & &\,. & &\,. & &\,. & &\,. & &\,. &
  &\,. & &\,. & &\,. & &\,. & &\,. & &\,. & &\,. & &\,. & &\,. & &\,. & &\,. & &\,. &
  &\,. & &\,. & &\,. & &\,. & &\,. & &\,. & &\,. & &  - & &\,. & &\,. & &\,. & &\,.
  \\[-2pt] 
  &\,. & &\,. & &\,. & &\,. & &\,. & &\,. & &\,. & &\,. & &\,. & &\,. & &\,. & &\,. &
  &\,. & &  + & &\,. & &\,. & &\,. & &\,. & &\,. & &\,. & &\,. & &\,. & &\,. & &\,. &
  &\,. & &\,. & &\,. & &\,. & &\,. & &\,. & &\,. & &\,. & &\,. & &\,. & &\,. & &\,. &
  &\,. & &\,. & &\,. & &\,. & &\,. & &\,. & &\,. & &\,. & &\,. & &\,. & &\,. & &\,. &
  &\,. & &\,. & &\,. & &\,. & &\,. & &\,. & &  - & &\,. & &\,. & &\,. & &\,. & &\,.
  \\[-2pt] 
  &\,. & &\,. & &\,. & &\,. & &\,. & &\,. & &\,. & &\,. & &\,. & &\,. & &\,. & &\,. &
  &\,. & &\,. & &  + & &\,. & &\,. & &\,. & &\,. & &\,. & &\,. & &\,. & &\,. & &\,. &
  &\,. & &\,. & &\,. & &\,. & &\,. & &\,. & &\,. & &\,. & &\,. & &\,. & &\,. & &\,. &
  &\,. & &\,. & &\,. & &\,. & &\,. & &\,. & &\,. & &\,. & &\,. & &\,. & &\,. & &\,. &
  &\,. & &\,. & &\,. & &\,. & &\,. & &\,. & &\,. & &\,. & &  - & &\,. & &\,. & &\,.
  \\[-2pt] 
  &\,. & &\,. & &\,. & &\,. & &\,. & &\,. & &\,. & &\,. & &\,. & &\,. & &\,. & &\,. &
  &\,. & &\,. & &\,. & &  + & &\,. & &\,. & &\,. & &\,. & &\,. & &\,. & &\,. & &\,. &
  &\,. & &\,. & &\,. & &\,. & &\,. & &\,. & &\,. & &\,. & &\,. & &\,. & &\,. & &\,. &
  &\,. & &\,. & &\,. & &\,. & &\,. & &\,. & &\,. & &\,. & &\,. & &\,. & &\,. & &\,. &
  &\,. & &\,. & &\,. & &\,. & &\,. & &\,. & &  - & &\,. & &\,. & &\,. & &\,. & &\,.
  \\[-2pt] 
  &\,. & &\,. & &\,. & &\,. & &\,. & &\,. & &\,. & &\,. & &\,. & &\,. & &\,. & &\,. &
  &\,. & &\,. & &\,. & &\,. & &  + & &\,. & &\,. & &\,. & &\,. & &\,. & &\,. & &\,. &
  &\,. & &\,. & &\,. & &\,. & &\,. & &\,. & &\,. & &\,. & &\,. & &\,. & &\,. & &\,. &
  &\,. & &\,. & &\,. & &\,. & &\,. & &\,. & &\,. & &\,. & &\,. & &\,. & &\,. & &\,. &
  &\,. & &\,. & &\,. & &\,. & &\,. & &\,. & &\,. & &\,. & &  - & &\,. & &\,. & &\,.
  \\[-2pt] 
  &\,. & &\,. & &\,. & &\,. & &\,. & &\,. & &\,. & &\,. & &\,. & &\,. & &\,. & &\,. &
  &\,. & &\,. & &\,. & &\,. & &\,. & &  + & &\,. & &\,. & &\,. & &\,. & &\,. & &\,. &
  &\,. & &\,. & &\,. & &\,. & &\,. & &\,. & &\,. & &\,. & &\,. & &\,. & &\,. & &\,. &
  &\,. & &\,. & &\,. & &\,. & &\,. & &\,. & &\,. & &\,. & &\,. & &\,. & &\,. & &\,. &
  &\,. & &\,. & &\,. & &\,. & &\,. & &\,. & &\,. & &  - & &\,. & &\,. & &\,. & &\,.
  \\[-2pt] 
  &\,. & &\,. & &\,. & &\,. & &\,. & &\,. & &\,. & &\,. & &\,. & &\,. & &\,. & &\,. &
  &\,. & &\,. & &\,. & &\,. & &\,. & &\,. & &  + & &\,. & &\,. & &\,. & &\,. & &\,. &
  &\,. & &\,. & &\,. & &\,. & &\,. & &\,. & &\,. & &\,. & &\,. & &\,. & &\,. & &\,. &
  &\,. & &\,. & &\,. & &\,. & &\,. & &\,. & &\,. & &\,. & &\,. & &\,. & &\,. & &\,. &
  &\,. & &\,. & &\,. & &\,. & &\,. & &\,. & &\,. & &\,. & &\,. & &\,. & &  - & &\,.
  \\[-2pt] 
  &\,. & &\,. & &\,. & &\,. & &\,. & &\,. & &\,. & &\,. & &\,. & &\,. & &\,. & &\,. &
  &\,. & &\,. & &\,. & &\,. & &\,. & &\,. & &\,. & &  + & &\,. & &\,. & &\,. & &\,. &
  &\,. & &\,. & &\,. & &\,. & &\,. & &\,. & &\,. & &\,. & &\,. & &\,. & &\,. & &\,. &
  &\,. & &\,. & &\,. & &\,. & &\,. & &\,. & &\,. & &\,. & &\,. & &\,. & &\,. & &\,. &
  &\,. & &\,. & &\,. & &\,. & &\,. & &\,. & &\,. & &\,. & &\,. & &  - & &\,. & &\,.
  \\[-2pt] 
  &\,. & &\,. & &\,. & &\,. & &\,. & &\,. & &\,. & &\,. & &\,. & &\,. & &\,. & &\,. &
  &\,. & &\,. & &\,. & &\,. & &\,. & &\,. & &\,. & &\,. & &  + & &\,. & &\,. & &\,. &
  &\,. & &\,. & &\,. & &\,. & &\,. & &\,. & &\,. & &\,. & &\,. & &\,. & &\,. & &\,. &
  &\,. & &\,. & &\,. & &\,. & &\,. & &\,. & &\,. & &\,. & &\,. & &\,. & &\,. & &\,. &
  &\,. & &\,. & &\,. & &\,. & &\,. & &\,. & &\,. & &\,. & &\,. & &\,. & &\,. & &  -
  \\[-2pt] 
  &\,. & &\,. & &\,. & &\,. & &\,. & &\,. & &\,. & &\,. & &\,. & &\,. & &\,. & &\,. &
  &\,. & &\,. & &\,. & &\,. & &\,. & &\,. & &\,. & &\,. & &\,. & &  + & &\,. & &\,. &
  &\,. & &\,. & &\,. & &\,. & &\,. & &\,. & &\,. & &\,. & &\,. & &\,. & &\,. & &\,. &
  &\,. & &\,. & &\,. & &\,. & &\,. & &\,. & &\,. & &\,. & &\,. & &\,. & &\,. & &\,. &
  &\,. & &\,. & &\,. & &\,. & &\,. & &\,. & &\,. & &\,. & &\,. & &  - & &\,. & &\,.
  \\[-2pt] 
  &\,. & &\,. & &\,. & &\,. & &\,. & &\,. & &\,. & &\,. & &\,. & &\,. & &\,. & &\,. &
  &\,. & &\,. & &\,. & &\,. & &\,. & &\,. & &\,. & &\,. & &\,. & &\,. & &  + & &\,. &
  &\,. & &\,. & &\,. & &\,. & &\,. & &\,. & &\,. & &\,. & &\,. & &\,. & &\,. & &\,. &
  &\,. & &\,. & &\,. & &\,. & &\,. & &\,. & &\,. & &\,. & &\,. & &\,. & &\,. & &\,. &
  &\,. & &\,. & &\,. & &\,. & &\,. & &\,. & &\,. & &\,. & &\,. & &\,. & &\,. & &  -
  \\[-2pt] 
  &\,. & &\,. & &\,. & &\,. & &\,. & &\,. & &\,. & &\,. & &\,. & &\,. & &\,. & &\,. &
  &\,. & &\,. & &\,. & &\,. & &\,. & &\,. & &\,. & &\,. & &\,. & &\,. & &\,. & &  + &
  &\,. & &\,. & &\,. & &\,. & &\,. & &\,. & &\,. & &\,. & &\,. & &\,. & &\,. & &\,. &
  &\,. & &\,. & &\,. & &\,. & &\,. & &\,. & &\,. & &\,. & &\,. & &\,. & &\,. & &\,. &
  &\,. & &\,. & &\,. & &\,. & &\,. & &\,. & &\,. & &\,. & &\,. & &\,. & &  - & &\,.
  \\[-2pt] 
  &\,. & &\,. & &\,. & &\,. & &\,. & &\,. & &\,. & &\,. & &\,. & &\,. & &\,. & &\,. &
  &\,. & &\,. & &\,. & &\,. & &\,. & &\,. & &\,. & &\,. & &\,. & &\,. & &\,. & &\,. &
  &  + & &\,. & &\,. & &\,. & &\,. & &\,. & &\,. & &\,. & &\,. & &\,. & &\,. & &\,. &
  &\,. & &\,. & &  + & &\,. & &\,. & &\,. & &\,. & &\,. & &\,. & &\,. & &\,. & &\,. &
  &  + & &  + & &  + & &\,. & &\,. & &\,. & &\,. & &\,. & &\,. & &\,. & &\,. & &\,.
  \\[-2pt] 
  &\,. & &\,. & &\,. & &\,. & &\,. & &\,. & &\,. & &\,. & &\,. & &\,. & &\,. & &\,. &
  &\,. & &\,. & &\,. & &\,. & &\,. & &\,. & &\,. & &\,. & &\,. & &\,. & &\,. & &\,. &
  &\,. & &  + & &\,. & &\,. & &\,. & &\,. & &\,. & &\,. & &\,. & &\,. & &\,. & &\,. &
  &\,. & &\,. & &  + & &\,. & &\,. & &\,. & &\,. & &\,. & &\,. & &\,. & &\,. & &\,. &
  &  + & &  + & &  + & &\,. & &\,. & &\,. & &\,. & &\,. & &\,. & &\,. & &\,. & &\,.
  \\[-2pt] 
  &\,. & &\,. & &\,. & &\,. & &\,. & &\,. & &\,. & &\,. & &\,. & &\,. & &\,. & &\,. &
  &\,. & &\,. & &\,. & &\,. & &\,. & &\,. & &\,. & &\,. & &\,. & &\,. & &\,. & &\,. &
  &\,. & &\,. & &  + & &\,. & &\,. & &\,. & &\,. & &\,. & &\,. & &\,. & &\,. & &\,. &
  &\,. & &\,. & &  + & &\,. & &\,. & &\,. & &\,. & &\,. & &\,. & &\,. & &\,. & &\,. &
  &  + & &  + & &  + & &\,. & &\,. & &\,. & &\,. & &\,. & &\,. & &\,. & &\,. & &\,.
  \\[-2pt] 
  &\,. & &\,. & &\,. & &\,. & &\,. & &\,. & &\,. & &\,. & &\,. & &\,. & &\,. & &\,. &
  &\,. & &\,. & &\,. & &\,. & &\,. & &\,. & &\,. & &\,. & &\,. & &\,. & &\,. & &\,. &
  &\,. & &\,. & &\,. & &  + & &\,. & &\,. & &\,. & &\,. & &\,. & &\,. & &\,. & &\,. &
  &\,. & &\,. & &\,. & &\,. & &\,. & &  + & &\,. & &\,. & &\,. & &\,. & &\,. & &\,. &
  &\,. & &\,. & &\,. & &  + & &  + & &  + & &\,. & &\,. & &\,. & &\,. & &\,. & &\,.
  \\[-2pt] 
  &\,. & &\,. & &\,. & &\,. & &\,. & &\,. & &\,. & &\,. & &\,. & &\,. & &\,. & &\,. &
  &\,. & &\,. & &\,. & &\,. & &\,. & &\,. & &\,. & &\,. & &\,. & &\,. & &\,. & &\,. &
  &\,. & &\,. & &\,. & &\,. & &  + & &\,. & &\,. & &\,. & &\,. & &\,. & &\,. & &\,. &
  &\,. & &\,. & &\,. & &\,. & &\,. & &  + & &\,. & &\,. & &\,. & &\,. & &\,. & &\,. &
  &\,. & &\,. & &\,. & &  + & &  + & &  + & &\,. & &\,. & &\,. & &\,. & &\,. & &\,.
  \\[-2pt] 
  &\,. & &\,. & &\,. & &\,. & &\,. & &\,. & &\,. & &\,. & &\,. & &\,. & &\,. & &\,. &
  &\,. & &\,. & &\,. & &\,. & &\,. & &\,. & &\,. & &\,. & &\,. & &\,. & &\,. & &\,. &
  &\,. & &\,. & &\,. & &\,. & &\,. & &  + & &\,. & &\,. & &\,. & &\,. & &\,. & &\,. &
  &\,. & &\,. & &\,. & &\,. & &\,. & &  + & &\,. & &\,. & &\,. & &\,. & &\,. & &\,. &
  &\,. & &\,. & &\,. & &  + & &  + & &  + & &\,. & &\,. & &\,. & &\,. & &\,. & &\,.
  \\[-2pt] 
  &\,. & &\,. & &\,. & &\,. & &\,. & &\,. & &\,. & &\,. & &\,. & &\,. & &\,. & &\,. &
  &\,. & &\,. & &\,. & &\,. & &\,. & &\,. & &\,. & &\,. & &\,. & &\,. & &\,. & &\,. &
  &\,. & &\,. & &\,. & &\,. & &\,. & &\,. & &  + & &\,. & &\,. & &\,. & &\,. & &\,. &
  &\,. & &\,. & &\,. & &\,. & &\,. & &\,. & &\,. & &\,. & &  + & &\,. & &\,. & &\,. &
  &\,. & &\,. & &\,. & &\,. & &\,. & &\,. & &  + & &  + & &  + & &\,. & &\,. & &\,.
  \\[-2pt] 
  &\,. & &\,. & &\,. & &\,. & &\,. & &\,. & &\,. & &\,. & &\,. & &\,. & &\,. & &\,. &
  &\,. & &\,. & &\,. & &\,. & &\,. & &\,. & &\,. & &\,. & &\,. & &\,. & &\,. & &\,. &
  &\,. & &\,. & &\,. & &\,. & &\,. & &\,. & &\,. & &  + & &\,. & &\,. & &\,. & &\,. &
  &\,. & &\,. & &\,. & &\,. & &\,. & &\,. & &\,. & &\,. & &  + & &\,. & &\,. & &\,. &
  &\,. & &\,. & &\,. & &\,. & &\,. & &\,. & &  + & &  + & &  + & &\,. & &\,. & &\,.
  \\[-2pt] 
  &\,. & &\,. & &\,. & &\,. & &\,. & &\,. & &\,. & &\,. & &\,. & &\,. & &\,. & &\,. &
  &\,. & &\,. & &\,. & &\,. & &\,. & &\,. & &\,. & &\,. & &\,. & &\,. & &\,. & &\,. &
  &\,. & &\,. & &\,. & &\,. & &\,. & &\,. & &\,. & &\,. & &  + & &\,. & &\,. & &\,. &
  &\,. & &\,. & &\,. & &\,. & &\,. & &\,. & &\,. & &\,. & &  + & &\,. & &\,. & &\,. &
  &\,. & &\,. & &\,. & &\,. & &\,. & &\,. & &  + & &  + & &  + & &\,. & &\,. & &\,.
  \\[-2pt] 
  &\,. & &\,. & &\,. & &\,. & &\,. & &\,. & &\,. & &\,. & &\,. & &\,. & &\,. & &\,. &
  &\,. & &\,. & &\,. & &\,. & &\,. & &\,. & &\,. & &\,. & &\,. & &\,. & &\,. & &\,. &
  &\,. & &\,. & &\,. & &\,. & &\,. & &\,. & &\,. & &\,. & &\,. & &  + & &\,. & &\,. &
  &\,. & &\,. & &\,. & &\,. & &\,. & &\,. & &\,. & &\,. & &\,. & &\,. & &\,. & &  + &
  &\,. & &\,. & &\,. & &\,. & &\,. & &\,. & &\,. & &\,. & &\,. & &  + & &  + & &  +
  \\[-2pt] 
  &\,. & &\,. & &\,. & &\,. & &\,. & &\,. & &\,. & &\,. & &\,. & &\,. & &\,. & &\,. &
  &\,. & &\,. & &\,. & &\,. & &\,. & &\,. & &\,. & &\,. & &\,. & &\,. & &\,. & &\,. &
  &\,. & &\,. & &\,. & &\,. & &\,. & &\,. & &\,. & &\,. & &\,. & &\,. & &  + & &\,. &
  &\,. & &\,. & &\,. & &\,. & &\,. & &\,. & &\,. & &\,. & &\,. & &\,. & &\,. & &  + &
  &\,. & &\,. & &\,. & &\,. & &\,. & &\,. & &\,. & &\,. & &\,. & &  + & &  + & &  +
  \\[-2pt] 
  &\,. & &\,. & &\,. & &\,. & &\,. & &\,. & &\,. & &\,. & &\,. & &\,. & &\,. & &\,. &
  &\,. & &\,. & &\,. & &\,. & &\,. & &\,. & &\,. & &\,. & &\,. & &\,. & &\,. & &\,. &
  &\,. & &\,. & &\,. & &\,. & &\,. & &\,. & &\,. & &\,. & &\,. & &\,. & &\,. & &  + &
  &\,. & &\,. & &\,. & &\,. & &\,. & &\,. & &\,. & &\,. & &\,. & &\,. & &\,. & &  + &
  &\,. & &\,. & &\,. & &\,. & &\,. & &\,. & &\,. & &\,. & &\,. & &  + & &  + & &  +
  \\[-2pt] 
  &\,. & &\,. & &\,. & &\,. & &\,. & &\,. & &\,. & &\,. & &\,. & &\,. & &\,. & &\,. &
  &\,. & &\,. & &\,. & &\,. & &\,. & &\,. & &\,. & &\,. & &\,. & &\,. & &\,. & &\,. &
  &\,. & &\,. & &\,. & &\,. & &\,. & &\,. & &\,. & &\,. & &\,. & &\,. & &\,. & &\,. &
  &  + & &\,. & &  - & &\,. & &\,. & &\,. & &\,. & &\,. & &\,. & &\,. & &\,. & &\,. &
  &\,. & &\,. & &\,. & &\,. & &\,. & &\,. & &\,. & &\,. & &\,. & &\,. & &\,. & &\,.
  \\[-2pt] 
  &\,. & &\,. & &\,. & &\,. & &\,. & &\,. & &\,. & &\,. & &\,. & &\,. & &\,. & &\,. &
  &\,. & &\,. & &\,. & &\,. & &\,. & &\,. & &\,. & &\,. & &\,. & &\,. & &\,. & &\,. &
  &\,. & &\,. & &\,. & &\,. & &\,. & &\,. & &\,. & &\,. & &\,. & &\,. & &\,. & &\,. &
  &\,. & &  + & &  - & &\,. & &\,. & &\,. & &\,. & &\,. & &\,. & &\,. & &\,. & &\,. &
  &\,. & &\,. & &\,. & &\,. & &\,. & &\,. & &\,. & &\,. & &\,. & &\,. & &\,. & &\,.
  \\[-2pt] 
  &\,. & &\,. & &\,. & &\,. & &\,. & &\,. & &\,. & &\,. & &\,. & &\,. & &\,. & &\,. &
  &\,. & &\,. & &\,. & &\,. & &\,. & &\,. & &\,. & &\,. & &\,. & &\,. & &\,. & &\,. &
  &\,. & &\,. & &\,. & &\,. & &\,. & &\,. & &\,. & &\,. & &\,. & &\,. & &\,. & &\,. &
  &\,. & &\,. & &\,. & &  + & &\,. & &  - & &\,. & &\,. & &\,. & &\,. & &\,. & &\,. &
  &\,. & &\,. & &\,. & &\,. & &\,. & &\,. & &\,. & &\,. & &\,. & &\,. & &\,. & &\,.
  \\[-2pt] 
  &\,. & &\,. & &\,. & &\,. & &\,. & &\,. & &\,. & &\,. & &\,. & &\,. & &\,. & &\,. &
  &\,. & &\,. & &\,. & &\,. & &\,. & &\,. & &\,. & &\,. & &\,. & &\,. & &\,. & &\,. &
  &\,. & &\,. & &\,. & &\,. & &\,. & &\,. & &\,. & &\,. & &\,. & &\,. & &\,. & &\,. &
  &\,. & &\,. & &\,. & &\,. & &  + & &  - & &\,. & &\,. & &\,. & &\,. & &\,. & &\,. &
  &\,. & &\,. & &\,. & &\,. & &\,. & &\,. & &\,. & &\,. & &\,. & &\,. & &\,. & &\,.
  \\[-2pt] 
  &\,. & &\,. & &\,. & &\,. & &\,. & &\,. & &\,. & &\,. & &\,. & &\,. & &\,. & &\,. &
  &\,. & &\,. & &\,. & &\,. & &\,. & &\,. & &\,. & &\,. & &\,. & &\,. & &\,. & &\,. &
  &\,. & &\,. & &\,. & &\,. & &\,. & &\,. & &\,. & &\,. & &\,. & &\,. & &\,. & &\,. &
  &\,. & &\,. & &\,. & &\,. & &\,. & &\,. & &  + & &\,. & &  - & &\,. & &\,. & &\,. &
  &\,. & &\,. & &\,. & &\,. & &\,. & &\,. & &\,. & &\,. & &\,. & &\,. & &\,. & &\,.
  \\[-2pt] 
  &\,. & &\,. & &\,. & &\,. & &\,. & &\,. & &\,. & &\,. & &\,. & &\,. & &\,. & &\,. &
  &\,. & &\,. & &\,. & &\,. & &\,. & &\,. & &\,. & &\,. & &\,. & &\,. & &\,. & &\,. &
  &\,. & &\,. & &\,. & &\,. & &\,. & &\,. & &\,. & &\,. & &\,. & &\,. & &\,. & &\,. &
  &\,. & &\,. & &\,. & &\,. & &\,. & &\,. & &\,. & &  + & &  - & &\,. & &\,. & &\,. &
  &\,. & &\,. & &\,. & &\,. & &\,. & &\,. & &\,. & &\,. & &\,. & &\,. & &\,. & &\,.
  \\[-2pt] 
  &\,. & &\,. & &\,. & &\,. & &\,. & &\,. & &\,. & &\,. & &\,. & &\,. & &\,. & &\,. &
  &\,. & &\,. & &\,. & &\,. & &\,. & &\,. & &\,. & &\,. & &\,. & &\,. & &\,. & &\,. &
  &\,. & &\,. & &\,. & &\,. & &\,. & &\,. & &\,. & &\,. & &\,. & &\,. & &\,. & &\,. &
  &\,. & &\,. & &\,. & &\,. & &\,. & &\,. & &\,. & &\,. & &\,. & &  + & &\,. & &  - &
  &\,. & &\,. & &\,. & &\,. & &\,. & &\,. & &\,. & &\,. & &\,. & &\,. & &\,. & &\,.
  \\[-2pt] 
  &\,. & &\,. & &\,. & &\,. & &\,. & &\,. & &\,. & &\,. & &\,. & &\,. & &\,. & &\,. &
  &\,. & &\,. & &\,. & &\,. & &\,. & &\,. & &\,. & &\,. & &\,. & &\,. & &\,. & &\,. &
  &\,. & &\,. & &\,. & &\,. & &\,. & &\,. & &\,. & &\,. & &\,. & &\,. & &\,. & &\,. &
  &\,. & &\,. & &\,. & &\,. & &\,. & &\,. & &\,. & &\,. & &\,. & &\,. & &  + & &  - &
  &\,. & &\,. & &\,. & &\,. & &\,. & &\,. & &\,. & &\,. & &\,. & &\,. & &\,. & &\,.
  \\
  \hline
  &\,. & &\,. & &\,. & &\,. & &\,. & &\,. & &\,. & &\,. & &\,. & &\,. & &\,. & &\,. &
  &\,. & &\,. & &\,. & &\,. & &\,. & &\,. & &\,. & &\,. & &\,. & &\,. & &\,. & &\,. &
  &  - & &  - & &  - & &\,. & &\,. & &\,. & &\,. & &\,. & &\,. & &\,. & &\,. & &\,. &
  &  + & &  + & &  + & &\,. & &\,. & &\,. & &\,. & &\,. & &\,. & &\,. & &\,. & &\,. &
  &\,. & &\,. & &\,. & &\,. & &\,. & &\,. & &\,. & &\,. & &\,. & &\,. & &\,. & &\,.
  \\[-2pt] 
  &\,. & &\,. & &\,. & &\,. & &\,. & &\,. & &\,. & &\,. & &\,. & &\,. & &\,. & &\,. &
  &\,. & &\,. & &\,. & &\,. & &\,. & &\,. & &\,. & &\,. & &\,. & &\,. & &\,. & &\,. &
  &\,. & &\,. & &\,. & &  - & &  - & &  - & &\,. & &\,. & &\,. & &\,. & &\,. & &\,. &
  &\,. & &\,. & &\,. & &  + & &  + & &  + & &\,. & &\,. & &\,. & &\,. & &\,. & &\,. &
  &\,. & &\,. & &\,. & &\,. & &\,. & &\,. & &\,. & &\,. & &\,. & &\,. & &\,. & &\,.
  \\[-2pt] 
  &\,. & &\,. & &\,. & &\,. & &\,. & &\,. & &\,. & &\,. & &\,. & &\,. & &\,. & &\,. &
  &\,. & &\,. & &\,. & &\,. & &\,. & &\,. & &\,. & &\,. & &\,. & &\,. & &\,. & &\,. &
  &\,. & &\,. & &\,. & &\,. & &\,. & &\,. & &  - & &  - & &  - & &\,. & &\,. & &\,. &
  &\,. & &\,. & &\,. & &\,. & &\,. & &\,. & &  + & &  + & &  + & &\,. & &\,. & &\,. &
  &\,. & &\,. & &\,. & &\,. & &\,. & &\,. & &\,. & &\,. & &\,. & &\,. & &\,. & &\,.
  \\[-2pt] 
  &\,. & &\,. & &\,. & &\,. & &\,. & &\,. & &\,. & &\,. & &\,. & &\,. & &\,. & &\,. &
  &\,. & &\,. & &\,. & &\,. & &\,. & &\,. & &\,. & &\,. & &\,. & &\,. & &\,. & &\,. &
  &\,. & &\,. & &\,. & &\,. & &\,. & &\,. & &\,. & &\,. & &\,. & &  - & &  - & &  - &
  &\,. & &\,. & &\,. & &\,. & &\,. & &\,. & &\,. & &\,. & &\,. & &  + & &  + & &  + &
  &\,. & &\,. & &\,. & &\,. & &\,. & &\,. & &\,. & &\,. & &\,. & &\,. & &\,. & &\,.
  \\[-2pt] 
  &\,. & &  + & &\,. & &  + & &\,. & &\,. & &\,. & &\,. & &\,. & &\,. & &\,. & &\,. &
  &\,. & &\,. & &\,. & &\,. & &\,. & &\,. & &\,. & &\,. & &\,. & &\,. & &\,. & &\,. &
  &  - & &  - & &  - & &\,. & &\,. & &\,. & &\,. & &\,. & &\,. & &\,. & &\,. & &\,. &
  &\,. & &\,. & &\,. & &\,. & &\,. & &\,. & &\,. & &\,. & &\,. & &\,. & &\,. & &\,. &
  &  + & &\,. & &\,. & &\,. & &\,. & &\,. & &\,. & &\,. & &\,. & &\,. & &\,. & &\,.
  \\[-2pt] 
  &  + & &\,. & &\,. & &\,. & &\,. & &  + & &\,. & &\,. & &\,. & &\,. & &\,. & &\,. &
  &\,. & &\,. & &\,. & &\,. & &\,. & &\,. & &\,. & &\,. & &\,. & &\,. & &\,. & &\,. &
  &  - & &  - & &  - & &\,. & &\,. & &\,. & &\,. & &\,. & &\,. & &\,. & &\,. & &\,. &
  &\,. & &\,. & &\,. & &\,. & &\,. & &\,. & &\,. & &\,. & &\,. & &\,. & &\,. & &\,. &
  &\,. & &  + & &\,. & &\,. & &\,. & &\,. & &\,. & &\,. & &\,. & &\,. & &\,. & &\,.
  \\[-2pt] 
  &\,. & &\,. & &  + & &\,. & &  + & &\,. & &\,. & &\,. & &\,. & &\,. & &\,. & &\,. &
  &\,. & &\,. & &\,. & &\,. & &\,. & &\,. & &\,. & &\,. & &\,. & &\,. & &\,. & &\,. &
  &  - & &  - & &  - & &\,. & &\,. & &\,. & &\,. & &\,. & &\,. & &\,. & &\,. & &\,. &
  &\,. & &\,. & &\,. & &\,. & &\,. & &\,. & &\,. & &\,. & &\,. & &\,. & &\,. & &\,. &
  &\,. & &\,. & &  + & &\,. & &\,. & &\,. & &\,. & &\,. & &\,. & &\,. & &\,. & &\,.
  \\[-2pt] 
  &\,. & &\,. & &\,. & &\,. & &\,. & &\,. & &\,. & &  + & &\,. & &  + & &\,. & &\,. &
  &\,. & &\,. & &\,. & &\,. & &\,. & &\,. & &\,. & &\,. & &\,. & &\,. & &\,. & &\,. &
  &\,. & &\,. & &\,. & &  - & &  - & &  - & &\,. & &\,. & &\,. & &\,. & &\,. & &\,. &
  &\,. & &\,. & &\,. & &\,. & &\,. & &\,. & &\,. & &\,. & &\,. & &\,. & &\,. & &\,. &
  &\,. & &\,. & &\,. & &  + & &\,. & &\,. & &\,. & &\,. & &\,. & &\,. & &\,. & &\,.
  \\[-2pt] 
  &\,. & &\,. & &\,. & &\,. & &\,. & &\,. & &  + & &\,. & &\,. & &\,. & &\,. & &  + &
  &\,. & &\,. & &\,. & &\,. & &\,. & &\,. & &\,. & &\,. & &\,. & &\,. & &\,. & &\,. &
  &\,. & &\,. & &\,. & &  - & &  - & &  - & &\,. & &\,. & &\,. & &\,. & &\,. & &\,. &
  &\,. & &\,. & &\,. & &\,. & &\,. & &\,. & &\,. & &\,. & &\,. & &\,. & &\,. & &\,. &
  &\,. & &\,. & &\,. & &\,. & &  + & &\,. & &\,. & &\,. & &\,. & &\,. & &\,. & &\,.
  \\[-2pt] 
  &\,. & &\,. & &\,. & &\,. & &\,. & &\,. & &\,. & &\,. & &  + & &\,. & &  + & &\,. &
  &\,. & &\,. & &\,. & &\,. & &\,. & &\,. & &\,. & &\,. & &\,. & &\,. & &\,. & &\,. &
  &\,. & &\,. & &\,. & &  - & &  - & &  - & &\,. & &\,. & &\,. & &\,. & &\,. & &\,. &
  &\,. & &\,. & &\,. & &\,. & &\,. & &\,. & &\,. & &\,. & &\,. & &\,. & &\,. & &\,. &
  &\,. & &\,. & &\,. & &\,. & &\,. & &  + & &\,. & &\,. & &\,. & &\,. & &\,. & &\,.
  \\[-2pt] 
  &\,. & &\,. & &\,. & &\,. & &\,. & &\,. & &\,. & &\,. & &\,. & &\,. & &\,. & &\,. &
  &\,. & &  + & &\,. & &  + & &\,. & &\,. & &\,. & &\,. & &\,. & &\,. & &\,. & &\,. &
  &\,. & &\,. & &\,. & &\,. & &\,. & &\,. & &  - & &  - & &  - & &\,. & &\,. & &\,. &
  &\,. & &\,. & &\,. & &\,. & &\,. & &\,. & &\,. & &\,. & &\,. & &\,. & &\,. & &\,. &
  &\,. & &\,. & &\,. & &\,. & &\,. & &\,. & &  + & &\,. & &\,. & &\,. & &\,. & &\,.
  \\[-2pt] 
  &\,. & &\,. & &\,. & &\,. & &\,. & &\,. & &\,. & &\,. & &\,. & &\,. & &\,. & &\,. &
  &  + & &\,. & &\,. & &\,. & &\,. & &  + & &\,. & &\,. & &\,. & &\,. & &\,. & &\,. &
  &\,. & &\,. & &\,. & &\,. & &\,. & &\,. & &  - & &  - & &  - & &\,. & &\,. & &\,. &
  &\,. & &\,. & &\,. & &\,. & &\,. & &\,. & &\,. & &\,. & &\,. & &\,. & &\,. & &\,. &
  &\,. & &\,. & &\,. & &\,. & &\,. & &\,. & &\,. & &  + & &\,. & &\,. & &\,. & &\,.
  \\[-2pt] 
  &\,. & &\,. & &\,. & &\,. & &\,. & &\,. & &\,. & &\,. & &\,. & &\,. & &\,. & &\,. &
  &\,. & &\,. & &  + & &\,. & &  + & &\,. & &\,. & &\,. & &\,. & &\,. & &\,. & &\,. &
  &\,. & &\,. & &\,. & &\,. & &\,. & &\,. & &  - & &  - & &  - & &\,. & &\,. & &\,. &
  &\,. & &\,. & &\,. & &\,. & &\,. & &\,. & &\,. & &\,. & &\,. & &\,. & &\,. & &\,. &
  &\,. & &\,. & &\,. & &\,. & &\,. & &\,. & &\,. & &\,. & &  + & &\,. & &\,. & &\,.
  \\[-2pt] 
  &\,. & &\,. & &\,. & &\,. & &\,. & &\,. & &\,. & &\,. & &\,. & &\,. & &\,. & &\,. &
  &\,. & &\,. & &\,. & &\,. & &\,. & &\,. & &\,. & &  + & &\,. & &  + & &\,. & &\,. &
  &\,. & &\,. & &\,. & &\,. & &\,. & &\,. & &\,. & &\,. & &\,. & &  - & &  - & &  - &
  &\,. & &\,. & &\,. & &\,. & &\,. & &\,. & &\,. & &\,. & &\,. & &\,. & &\,. & &\,. &
  &\,. & &\,. & &\,. & &\,. & &\,. & &\,. & &\,. & &\,. & &\,. & &  + & &\,. & &\,.
  \\[-2pt] 
  &\,. & &\,. & &\,. & &\,. & &\,. & &\,. & &\,. & &\,. & &\,. & &\,. & &\,. & &\,. &
  &\,. & &\,. & &\,. & &\,. & &\,. & &\,. & &  + & &\,. & &\,. & &\,. & &\,. & &  + &
  &\,. & &\,. & &\,. & &\,. & &\,. & &\,. & &\,. & &\,. & &\,. & &  - & &  - & &  - &
  &\,. & &\,. & &\,. & &\,. & &\,. & &\,. & &\,. & &\,. & &\,. & &\,. & &\,. & &\,. &
  &\,. & &\,. & &\,. & &\,. & &\,. & &\,. & &\,. & &\,. & &\,. & &\,. & &  + & &\,.
  \\[-2pt] 
  &\,. & &\,. & &\,. & &\,. & &\,. & &\,. & &\,. & &\,. & &\,. & &\,. & &\,. & &\,. &
  &\,. & &\,. & &\,. & &\,. & &\,. & &\,. & &\,. & &\,. & &  + & &\,. & &  + & &\,. &
  &\,. & &\,. & &\,. & &\,. & &\,. & &\,. & &\,. & &\,. & &\,. & &  - & &  - & &  - &
  &\,. & &\,. & &\,. & &\,. & &\,. & &\,. & &\,. & &\,. & &\,. & &\,. & &\,. & &\,. &
  &\,. & &\,. & &\,. & &\,. & &\,. & &\,. & &\,. & &\,. & &\,. & &\,. & &\,. & &  +
  \end{alignat*}
  \end{center}
  \caption{The row canonical form and the nullspace basis of $E_4$}
  \label{degree4expansionrcfnull}
  \end{table}

  \begin{table}
  \[
  \begin{array}{r}
  - (a(bc))d - (a(bd))c - (a(cd))b + (ab)(cd) + (ac)(bd) + (ad)(bc) \\
  - (b(ac))d - (b(ad))c - (b(cd))a + (ba)(cd) + (bc)(ad) + (bd)(ac) \\
  - (c(ab))d - (c(ad))b - (c(bd))a + (ca)(bd) + (cb)(ad) + (cd)(ab) \\
  - (d(ab))c - (d(ac))b - (d(bc))a + (da)(bc) + (db)(ac) + (dc)(ab) \\
  ((ab)d)c + ((ac)d)b - (a(bc))d - (a(bd))c - (a(cd))b + a((bc)d) \\
  ((ab)c)d + ((ad)c)b - (a(bc))d - (a(bd))c - (a(cd))b + a((bd)c) \\
  ((ac)b)d + ((ad)b)c - (a(bc))d - (a(bd))c - (a(cd))b + a((cd)b) \\
  ((ba)d)c + ((bc)d)a - (b(ac))d - (b(ad))c - (b(cd))a + b((ac)d) \\
  ((ba)c)d + ((bd)c)a - (b(ac))d - (b(ad))c - (b(cd))a + b((ad)c) \\
  ((bc)a)d + ((bd)a)c - (b(ac))d - (b(ad))c - (b(cd))a + b((cd)a) \\
  ((ca)d)b + ((cb)d)a - (c(ab))d - (c(ad))b - (c(bd))a + c((ab)d) \\
  ((ca)b)d + ((cd)b)a - (c(ab))d - (c(ad))b - (c(bd))a + c((ad)b) \\
  ((cb)a)d + ((cd)a)b - (c(ab))d - (c(ad))b - (c(bd))a + c((bd)a) \\
  ((da)c)b + ((db)c)a - (d(ab))c - (d(ac))b - (d(bc))a + d((ab)c) \\
  ((da)b)c + ((dc)b)a - (d(ab))c - (d(ac))b - (d(bc))a + d((ac)b) \\
  ((db)a)c + ((dc)a)b - (d(ab))c - (d(ac))b - (d(bc))a + d((bc)a)
  \end{array}
  \]
  \caption{Nullspace basis identities 1--16}
  \label{nullspacebasis}
  \end{table}

\begin{remark} \label{nonlinearK}
If we write $D_{a,d}(x) = (a,x,d) = (ax)d - a(xd)$ then $K$ represents a form
of the derivation rule: $D_{a,d}(bc) = D_{a,d}(b)c + D_{a,d}(c)b$. In $K$, if
we (simultaneously) replace $a$ by $b$, $d$ by $c$, and $b, c$ by $a$, then we
obtain
  \allowdisplaybreaks
  \begin{align*}
  &
    ( ( b a ) c ) a
  + ( ( b a ) c ) a
  - ( b a^2 ) c
  - ( b ( a c ) ) a
  - ( b ( a c ) ) a
  + b ( a^2 c )
  \\
  &=
  2 ( ( b a ) c ) a - 2 ( b ( a c ) ) a - ( b a^2 ) c + b ( a^2 c )
  =
  2 (b,a,c) a - ( b, a^2, c ),
  \end{align*}
If we linearize $2(a,b,d)b - (a,b^2,d)$ and apply right commutativity then we
obtain an identity equivalent to $K$ (since $\mathrm{char}\,\mathbb{F} \ne 2$):
  \allowdisplaybreaks
  \begin{align*}
  &
  2 (a,b,d)c + 2 (a,c,d)b - (a,bc,d) - (a,cb,d)
  \\
  &
  =
  2 ((ab)d)c - 2 (a(bd))c + 2 ((ac)d)b - 2 (a(cd))b - (a(bc))d + a((bc)d)
  \\
  &\quad
  - (a(cb))d + a((cb)d)
  \\
  &=
  2 ((ab)d)c + 2 ((ac)d)b - 2 (a(bc))d - 2 (a(bd))c - 2 (a(cd))b + 2 a((bc)d).
  \end{align*}
\end{remark}

\begin{remark} \label{directproof}
We have the following direct proof that $(b,a^2,c) = 2(b,a,c)a$:
  \allowdisplaybreaks
  \begin{align*}
  &
  (b,a^2,c)
  =
  ( b a^2 ) c - b ( a^2 c )
  \\
  &=
  ( b \dashv ( a \dashv a + a \vdash a ) + ( a \dashv a + a \vdash a ) \vdash b ) \dashv c
  \\
  &\quad
  +
  c \vdash ( b \dashv ( a \dashv a + a \vdash a ) + ( a \dashv a + a \vdash a ) \vdash b )
  \\
  &\quad
  -
  b \dashv ( ( a \dashv a + a \vdash a ) \dashv c + c \vdash ( a \dashv a + a \vdash a ) )
  \\
  &\quad
  -
  ( ( a \dashv a + a \vdash a ) \dashv c + c \vdash ( a \dashv a + a \vdash a ) ) \vdash b
  \\
  &=
  ( b \dashv ( a \dashv a ) ) \dashv c
  +
  ( b \dashv ( a \vdash a ) ) \dashv c
  +
  ( ( a \dashv a ) \vdash b ) \dashv c
  +
  ( ( a \vdash a ) \vdash b ) \dashv c
  \\
  &\quad
  +
  c \vdash ( b \dashv ( a \dashv a ) )
  +
  c \vdash ( b \dashv ( a \vdash a ) )
  +
  c \vdash ( ( a \dashv a ) \vdash b )
  +
  c \vdash ( ( a \vdash a ) \vdash b )
  \\
  &\quad
  -
  b \dashv ( ( a \dashv a ) \dashv c )
  -
  b \dashv ( ( a \vdash a ) \dashv c )
  -
  b \dashv ( c \vdash ( a \dashv a ) )
  -
  b \dashv ( c \vdash ( a \vdash a ) )
  \\
  &\quad
  -
  ( ( a \dashv a ) \dashv c ) \vdash b
  -
  ( ( a \vdash a ) \dashv c ) \vdash b
  -
  ( c \vdash ( a \dashv a ) ) \vdash b
  -
  ( c \vdash ( a \vdash a ) ) \vdash b
  \\
  &=
  2 \widehat b a a c
  +
  2 a a \widehat b c
  +
  2 c \widehat b a a
  +
  2 c a a \widehat b
  -
  2 \widehat b a a c
  -
  2 \widehat b c a a
  -
  2 a a c \widehat b
  -
  2 c a a \widehat b
  \\
  &=
  -
  2 \widehat b c a a
  +
  2 c \widehat b a a
  +
  2 a a \widehat b c
  -
  2 a a c \widehat b,
  \\
  &
  (b,a,c)a
  =
  ((ba)c)a - (b(ac))a
  \\
  &=
  ( ( b \dashv a + a \vdash b ) \dashv c + c \vdash ( b \dashv a + a \vdash b ) ) \dashv a
  \\
  &
  +
  a \vdash ( ( b \dashv a + a \vdash b ) \dashv c + c \vdash ( b \dashv a + a \vdash b ) )
  \\
  &\quad
  -
  ( b \dashv ( a \dashv c + c \vdash a ) + ( a \dashv c + c \vdash a ) \vdash b ) \dashv a
  \\
  &
  -
  a \vdash ( b \dashv ( a \dashv c + c \vdash a ) + ( a \dashv c + c \vdash a ) \vdash b )
  \\
  &=
  ( ( b \dashv a ) \dashv c ) \dashv a
  +
  ( ( a \vdash b ) \dashv c ) \dashv a
  +
  ( c \vdash ( b \dashv a ) ) \dashv a
  +
  ( c \vdash ( a \vdash b ) ) \dashv a
  \\
  &\quad
  +
  a \vdash ( ( b \dashv a ) \dashv c )
  +
  a \vdash ( ( a \vdash b ) \dashv c )
  +
  a \vdash ( c \vdash ( b \dashv a ) )
  +
  a \vdash ( c \vdash ( a \vdash b ) )
  \\
  &\quad
  -
  ( b \dashv ( a \dashv c ) ) \dashv a
  -
  ( b \dashv ( c \vdash a ) ) \dashv a
  -
  ( ( a \dashv c ) \vdash b ) \dashv a
  -
  ( ( c \vdash a ) \vdash b ) \dashv a
  \\
  &\quad
  -
  a \vdash ( b \dashv ( a \dashv c ) )
  -
  a \vdash ( b \dashv ( c \vdash a ) )
  -
  a \vdash ( ( a \dashv c ) \vdash b )
  -
  a \vdash ( ( c \vdash a ) \vdash b )
  \\
  &=
  \widehat b a c a
  +
  a \widehat b c a
  +
  c \widehat b a a
  +
  c a \widehat b a
  +
  a \widehat b a c
  +
  a a \widehat b c
  +
  a c \widehat b a
  +
  a c a \widehat b
  \\
  &\quad
  -
  \widehat b a c a
  -
  \widehat b c a a
  -
  a c \widehat b a
  -
  c a \widehat b a
  -
  a \widehat b a c
  -
  a \widehat b c a
  -
  a a c \widehat b
  -
  a c a \widehat b
  \\
  &=
  -
  \widehat b c a a
  +
  c \widehat b a a
  +
  a a \widehat b c
  -
  a a c \widehat b.
  \end{align*}
\end{remark}

\begin{definition} \label{newdefinition}
Vel\'asquez and Felipe \cite{VelasquezFelipe1} define a \textbf{(right) quasi-Jordan
algebra} to be a nonassociative algebra satisfying the right-commutative identity
$a ( b c ) = a ( c b )$ and the right quasi-Jordan identity $( b a^2 ) a = ( b a ) a^2$.
We propose the following definition which includes the (nonlinear version of the)
identity $K$: a \textbf{semispecial} (right) quasi-Jordan algebra over a field of
characteristic $\ne 2, 3$ is a (right) quasi-Jordan algebra satisfying the
\textbf{associator-derivation identity}:
  \[
  ( b, a^2, c ) = 2 (b,a,c) a.
  \]
\end{definition}

\begin{remark}
If $D$ is an (associative) dialgebra, then the \textbf{plus algebra} of $D$ is the algebra
$D^+$ with the same underlying vector space but with the operation
$ab = a \dashv b + b \vdash a$.
By analogy with the theory of Jordan algebras, the natural definition of a \textbf{special}
(right) quasi-Jordan algebra is one which is isomorphic to a subalgebra of $D^+$ for
some (associative) dialgebra $D$.  By the results of this paper, it is clear that every
special (right) quasi-Jordan algebra is a semispecial (right) quasi-Jordan algebra.
However, the converse is false: it has been shown recently by Bremner and Peresi
\cite{BremnerPeresi} that there exist polynomial identities in degree 8 satisfied by
the quasi-Jordan product in every (associative) dialgebra which do not follow from
the three identities of Definition \ref{newdefinition}. We therefore have three classes
of algebras, related by strict containments, where RQJ means (right) quasi-Jordan:
  \[
  \text{special RQJ algebras}
  \subsetneq
  \text{semispecial RQJ algebras}
  \subsetneq
  \text{RQJ algebras}
  \]
This motivates our choice of the term semispecial in Definition \ref{newdefinition}.
\end{remark}

\begin{remark}
Any associative algebra is a dialgebra in which the operations $\dashv$ and $\vdash$ coincide.
It follows from Remark \ref{directproof} that every special Jordan algebra is a
special quasi-Jordan algebra and hence satisfies the associator-derivation identity.
The referee raised the question of whether every Jordan algebra satisfies
the associator-derivation identity. The linearization of the Jordan identity $(a^2 b)a -
a^2(ba)$, assuming commutativity and characteristic $\ne 2$, is
  \[
  J(a,b,c,d)
  =
  ((ac)b)d + ((ad)b)c + ((cd)b)a - (ab)(cd) - (ac)(bd) - (ad)(bc).
  \]
The linearization of the associator-derivation identity $( b, a^2, c ) = 2 (b,a,c) a$,
assuming commutativity and characteristic $\ne 2$, is
  \[
  K(a,b,c,d)
  =
  ((ab)c)d - ((ac)b)d - ((ad)b)c + ((ad)c)b + ((bd)c)a - ((cd)b)a.
  \]
Then we have $K(a,b,c,d) = J(a,c,b,d) - J(a,b,c,d)$, which implies that every Jordan
algebra satisfies the associator-derivation identity.
Hentzel and Peresi  \cite{HentzelPeresi} consider the identity $(b,a^2,a) = 2 (b,a,a)a$
which is a weak form of the associator-derivation identity obtained by setting $c = a$.
They prove that over a field of characteristic $\ne 2, 3$ any commutative algebra which
satisfies this weaker identity, and is also semiprime, must be a Jordan algebra.
\end{remark}

\begin{remark}
After this work was completed, the author became aware of two closely related papers.
In Kolesnikov \cite{Kolesnikov} the polynomial identities (26--27) are equivalent
(collectively) to the opposite versions of the right quasi-Jordan identity and the
associator-derivation identity, $a ( a^2 b ) = a^2 ( a b )$ and $(b, a^2, c) = 2 a (b,a,c)$,
assuming left commutativity $(ab)c = (ba)c$. Pozhidaev \cite{Pozhidaev} gives a simple
algorithm to obtain from any variety of algebras a corresponding variety of dialgebras; 
in the case of Jordan algebras we obtain the variety of semispecial quasi-Jordan
algebras.
\end{remark}

\noindent \textbf{Acknowledgements}. I thank Ra\'ul Felipe, Luiz Peresi, Alexandre
Pozhidaev and the referee for very helpful comments.  This research was partially
supported by NSERC, the Natural Sciences and Engineering Research Council of
Canada.

\end{document}